\documentclass[11pt]{amsart}

\usepackage{amsaddr}
\usepackage{amssymb}
\usepackage{hyperref}

\usepackage{a4wide}

\title{Epsilon-strongly Graded Rings, Separability and Semisimplicity}

\newtheorem{thm}{Theorem}
\newtheorem{prop}[thm]{Proposition}
\newtheorem{lem}[thm]{Lemma}
\newtheorem{cor}[thm]{Corollary}
\theoremstyle{definition}
\newtheorem{defi}[thm]{Definition}
\newtheorem{exa}[thm]{Example}
\newtheorem{rem}[thm]{Remark}

\newcommand{\m}{{}^{-1}}

\newcommand{\ep}{\epsilon}

\DeclareMathOperator{\image}{Im}
\DeclareMathOperator{\Supp}{Supp}

\begin{document}

\author{Patrik Nystedt}
\address{Department of Engineering Science,
University West, 
SE-46186 Trollh\"{a}ttan, Sweden}

\author{Johan \"{O}inert}
\address{Department of Mathematics and Natural Sciences,
Blekinge Institute of Technology,
SE-37179 Karlskrona, Sweden}

\author{H\'{e}ctor Pinedo}
\address{Escuela de Matem\'{a}ticas, Universidad Industrial de Santander,
Carrera 27 Calle 9,
Edificio Camilo Torres
Apartado de correos 678,
Bucaramanga, Colombia}

\email{patrik.nystedt@hv.se; johan.oinert@bth.se; hpinedot@uis.edu.co}



\begin{abstract}
We introduce the class of epsilon-strongly graded rings
and show that it properly contains both the class of strongly graded
rings and the class of unital partial crossed products.
We determine precisely when an epsilon-strongly graded ring is separable over its principal component.
Thereby, we simultaneously generalize a result for 
strongly group graded rings by N{\v a}st{\v a}sescu, Van den Bergh and Van Oystaeyen,
and a result for unital partial crossed products by
Bagio, Lazzarin and Paques.
We also show that the class of unital partial crossed products
appear in the class of epsilon-strongly graded rings
in a fashion similar to how the classical crossed products
present themselves in the class of strongly graded rings.
Thereby, we obtain, in the special case of unital 
partial crossed products, a short proof of
a general result by Dokuchaev, Exel and Sim\'{o}n 
concerning when graded rings can be presented as partial crossed products.
We also provide some interesting classes of examples of separable 
epsilon-strongly graded rings, with finite as well as infinite grading groups.
In particular, we obtain an answer to a question raised by
Le Bruyn, Van den Bergh and Van Oystaeyen in 1988.
\end{abstract}

\maketitle

\section{Introduction}

Let $S$ be an associative ring equipped with a 
non-zero multiplicative identity element $1$.
Let $S/R$ be a ring extension.
By this we mean that $R$ is a subring of 
$S$ containing $1$.
Recall that $S/R$ is called {\it separable} if
the multiplication map $m : S \otimes_R S \rightarrow S$
is a splitting epimorphism of $R$-bimodules.
Equivalently, this can be formulated by saying that
there is $x \in S \otimes_R S$ satisfying $m(x)=1$
and that, for every $s \in S$, the relation $sx = xs$ holds.
In that case, $x$ is called a {\it separability element} of $S \otimes_R S$.
Separable ring extensions are a natural
generalization of the classical separability condition for
algebras over fields which in turn is a generalization of
separable field extensions (see e.g. \cite{dem}).
N{\v a}st{\v a}sescu, Van den Bergh and Van Oystaeyen \cite{nas89} 
have generalized this even further by introducing the notion of a {\it separable functor}.
They show that a ring extension is separable precisely when the
associated restriction functor is separable.
A lot of work has been devoted to the question of when ring
extensions are separable (see e.g. 
\cite{bag},
\cite{Bruyn1988}, 
\cite{cae},
\cite{cas}, 
\cite{dem}, 
\cite{hj}, 
\cite{hp}, 
\cite{L05},
\cite{L06},
\cite{miyashita},
\cite{nas89},
\cite{raf} and
\cite{theo}).
One reason for this intense interest is that some properties of the 
ground ring $R$ automatically are inherited by $S$,
such as semisimplicity and hereditarity (see e.g. \cite{nas89}).

In the context of group graded rings, necessary and sufficient
criteria for separability has been obtained in two different cases
(see Theorem~\ref{separabilitystrong} 
and Theorem~\ref{separabilitypartial} below).
Indeed, let $G$ be a group with identity element $e$.
Let $S$ be {\it graded} by $G$. Recall that this means that, for all $g,h \in G$, there 
is an additive subgroup $S_g$ of $S$
such that $S = \oplus_{g \in G} S_g$ and
$S_g S_h \subseteq S_{gh}$. 
The subring $R = S_e$ is called the {\it principal component} of $S$.

In the first case, $S$ is {\it strongly graded}. 
Recall that this means that 
$S_g S_h = S_{gh}$, for all $g,h \in G$.
This makes each $S_g$, for $g \in G$, an
invertible $R$-bimodule which implies that 
there is a unique ring automorphism $\beta_g : Z(R) \rightarrow Z(R)$
such that $\beta_g(r) s = s r$, for $r \in Z(R)$ and $s \in S_g$
(see \cite{miyashita} or e.g. Definition~\ref{definitiongamma} and Proposition~\ref{restriction}).  
If $G$ is finite, 
then the trace function ${\rm tr}_{\beta} : Z(R) \rightarrow Z(R)$
is defined by ${\rm tr}_{\beta}(r) = \sum_{g \in G} \beta_g(r)$, 
for $r \in Z(R)$. 

\begin{thm}[N{\v a}st{\v a}sescu, Van den Bergh and Van Oystaeyen \cite{nas89}]\label{separabilitystrong}
If $S$ is strongly graded by $G$,
then $S/R$ is separable if and only if $G$ is finite
and $1 \in {\rm tr}_{\beta}(Z(R))$.
\end{thm}

In the second case, $S$ is a {\it unital partial crossed product} of $G$ over $R$.
Recall that a {\it unital twisted partial action} of $G$ on $R$ is a triple 
\begin{displaymath}
\alpha = 
( \{ D_g \}_{g \in G} , \{ \alpha_g \}_{g \in G} , \{ w_{g,h} \}_{(g,h) \in G \times G} )
\end{displaymath}
where for each $g \in G$, $D_g$ is a unital ideal of $R$
having an (not necessarily non-zero) 
identity element $1_g$ which is central in $R$,
$\alpha_g : D_{g^{-1}} \rightarrow D_g$ is an isomorphism
of rings, and for each $(g,h) \in G \times G$,
$w_{g,h}$ is an invertible element from $D_g D_{gh}$,
satisfying the following assertions for all $g,h,l \in G$:
\begin{itemize}

\item[(P1)] $\alpha_e = {\rm id}_R$;

\item[(P2)] $\alpha_g(D_{g^{-1}} D_h) = D_g D_{gh}$;

\item[(P3)] if $r \in D_{h^{-1}} D_{(gh)^{-1}}$, 
then $\alpha_g ( \alpha_h (r) ) =
w_{g,h} \alpha_{gh}(r) w_{g,h}^{-1}$;

\item[(P4)] $w_{e,g} = w_{g,e} = 1_g$;

\item[(P5)] if $r \in D_{g^{-1}} D_h D_{hl}$, then
$\alpha_g(r w_{h,l}) w_{g,hl} =
\alpha_g(r) w_{g,h} w_{gh,l}$.

\end{itemize}
Given a unital twisted partial action of $G$ on $R$,
the unital partial crossed product $R \star_{\alpha}^w G$
is the direct sum $\oplus_{g \in G} D_g \delta_g$,
in which the $\delta_g$'s are formal symbols, and the 
multiplication is defined by the biadditive 
extension of the relations
\begin{itemize}

\item[(P6)] $(r \delta_g) (r' \delta_h) = r \alpha_g(r' 1_{g^{-1}}) w_{g,h} \delta_{gh}$,

\end{itemize}
for $g,h \in G$, $r \in D_g$ and $r' \in D_h$.  
If $G$ is finite, then the trace function
${\rm tr}_{\alpha} : Z(R) \rightarrow Z(R)$
is defined by 
${\rm tr}_{\alpha}(r) = \sum_{g \in G} \alpha_g(r 1_{g^{-1}})$, for $r \in Z(R)$.

\begin{thm}[Bagio, Lazzarin and Paques \cite{bag}]\label{separabilitypartial}
If $S$ is a unital partial crossed product 
of a finite group $G$ over $R$, then $S/R$ is 
separable if and only if $1 \in {\rm tr}_{\alpha}(Z(R))$. 
\end{thm}

In this article, we wish to unify Theorem~\ref{separabilitystrong}
and Theorem~\ref{separabilitypartial} (see Theorem~\ref{maintheorem})
for a class of rings which properly contains both the 
class of strongly graded rings and the class of partial crossed products.
We call this class of rings {\it epsilon-strongly graded}.
The term ``epsilon-strongly'' is supposed to be suggestive 
of the fact that the grading is ``an epsilon away'' from being strong.
Let us briefly describe the idea behind this class of rings.
Suppose that $S$ is a ring graded by a group $G$  and take $g,h \in G$.
Instead of postulating that $S_g S_h = S_{gh}$, 
as in the strongly graded case,
we relax this condition by saying that $S_g S_{g^{-1}}$
and $S_{h^{-1}} S_h$ are unital ideals of $R$ such that the equalities 
$S_g S_h = S_g S_{g^{-1}} S_{gh} = S_{gh} S_{h^{-1}} S_h$ hold.
The multiplicative identity element in $S_g S_{g^{-1}}$ is denoted by $\epsilon_g$.
Here is an outline of the article.

In Section~\ref{sectioncharacterization},
we introduce epsilon-strongly graded rings (see Definition~\ref{definitionepsilon}) and 
we give several equivalent characterizations of them (see Proposition~\ref{epsilon1}).

In Section~\ref{sectionseparability}, we show that if $S$ is epsilon-strongly
graded by $G$, then we can define a 
trace function ${\rm tr}_{\gamma} : Z(R)_{\rm fin} \rightarrow Z(R)$
(see Definition~\ref{definitiontrace})
which generalizes the trace functions from both the strongly graded case and
the partial crossed product situation.
Here $Z(R)_{\rm fin}$ is the set of $r \in Z(R)$
with the property that for all but finitely many $g \in G$,
the relation $r \epsilon_g = 0$ holds.
At the end of Section~\ref{sectionseparability},
we show the following simultaneous generalization
of Theorem~\ref{separabilitystrong} and Theorem~\ref{separabilitypartial}.
Notice that our result holds for any, possibly infinite, group $G$.

\begin{thm}\label{maintheorem}
If $S$ is epsilon-strongly graded by $G$,
then $S/R$ is separable if and only if $1 \in {\rm tr}_{\gamma}(Z(R)_{\rm fin})$. 
\end{thm}

In Section~\ref{sectionsemisimplicity}, 
we use Theorem~\ref{maintheorem}
to find criteria for when epsilon-strongly graded rings
are semisimple, hereditary or Frobenius
(see Theorem~\ref{theoremhereditary} and
Theorem~\ref{theoremfrobenius}).

In Section~\ref{simplicity}, we show that a result concerning simplicity
for strongly graded rings from \cite[Theorem 6.6]{O09}
can be generalized to epsilon-strongly graded rings
(see Proposition~\ref{Prop:Simplicity}).

In Section~\ref{sectionpartialcrossedproducts}, 
we introduce epsilon-crossed products
(see Definition~\ref{epsiloncrossedproduct}).
We show that the class of epsilon-crossed products
coincides with the class of unital partial crossed products
(see Theorem~\ref{correspondence}).
This is an epsilon-analogue of how the classical
crossed products appear in the class 
of strongly graded rings (see e.g. \cite{nas82}).
Thereby, we obtain, in the special case of unital 
partial crossed products, a short proof of
a general result by Dokuchaev, Exel and Sim\'{o}n \cite{dokuchaev2008}
concerning when graded rings can be presented as partial crossed products.
At the end of the section, we use Theorem~\ref{maintheorem}
to reformulate Theorem~\ref{separabilitypartial} so that
it holds for any, possibly infinite, group $G$
(see Theorem~\ref{genseparabilitypartial}).

In Section~\ref{sectionexample}, 
we provide a class of examples of separable 
epsilon-strongly $\mathbb{Z}_2$-graded rings, neither of which are strongly graded, 
nor partial crossed products, in any natural way
(see Proposition~\ref{propexample} and Proposition~\ref{specialcase}).
Thereby, we provide the first known non-trivial example of a ring, graded by a finite group, which is separable over its principal component but yet not strongly graded (see Remark~\ref{rem:Bruynbook} and \cite[Remark II.5.1.6]{Bruyn1988}).

In Section~\ref{Sec:MoritaRing},
we consider Morita rings 
which are in a natural way $\mathbb{Z}$-graded.
We show that, under weak assumptions,
they are in fact epsilon-strongly graded and separable over their principal components
(see Proposition~\ref{Prop:MoritaRingSeparability}).

\section{Some Characterizations of Epsilon-strongly Graded Rings}\label{sectioncharacterization}

In this section, we introduce epsilon-strongly graded rings (see Definition~\ref{definitionepsilon}) and 
we give several equivalent characterizations of them (see Proposition~\ref{epsilon1}).
Throughout the rest of this article, unless otherwise stated, let $G$ be an arbitrary group with identity element $e$.
In this section, let $S$ be an arbitrary unital ring which is graded by $G$ and put $R = S_e$.

\begin{defi}\label{definitionepsilon}
Let $S$ be a ring which is graded by $G$.
We say that $S$ is {\it epsilon-strongly graded by} $G$
if for each $g \in G$, $S_g S_{g^{-1}}$ is a unital ideal of $R$
such that for all $g,h \in G$ the equalities 
$S_g S_h = S_g S_{g^{-1}} S_{gh} = S_{gh} S_{h^{-1}} S_h$ hold.
In that case, for each $g \in G$, we let $\epsilon_g$ denote
the multiplicative identity element of $S_g S_{g^{-1}}$.
\end{defi}

\begin{prop}
If $S$ is epsilon-strongly graded by $G$, then, for every $g \in G$, $\epsilon_g \in Z(R)$.
\end{prop}

\begin{proof}
Take $g \in G$ and $r \in R$.
Since $S_g S_{g^{-1}}$ is an $R$-ideal it follows that
$\epsilon_g r , r \epsilon_g \in S_g S_{g^{-1}}$.
Using that $\epsilon_g$ is a multiplicative identity element of $S_g S_{g^{-1}}$,
we therefore get that $\epsilon_g r = ( \epsilon_g r ) \epsilon_g =
\epsilon_g (r \epsilon_g ) = r \epsilon_g$.
\end{proof}

\begin{defi}
Following \cite[Definition 4.5]{CEP2016} we say that $S$ is \emph{symmetrically graded} 
if for every $g \in G$, the equality $S_g S_{g^{-1}} S_g = S_g$ holds.
\end{defi}

\begin{prop}\label{epsilon1} 
The following assertions are equivalent:
\begin{enumerate}

\item[(i)] $S$ is epsilon-strongly graded by $G$;

\item[(ii)] $S$ is symmetrically graded by $G$, and for every $g \in G$ the $R$-ideal
$S_g S_{g^{-1}}$ is unital;

\item[(iii)] For every $g \in G$ there is an element $\epsilon_g \in S_g S_{g^{-1}}$
such that for all $s \in S_g$ the relations $\epsilon_g s = s = s \epsilon_{g^{-1}}$ hold;

\item[(iv)] For every $g \in G$ the left $R$-module $S_g$ is finitely generated
and projective, and the map $n_g : (S_g)_R \rightarrow {\rm Hom}_R( {}_R S_{g^{-1}} , R )_R$,
defined by $n_g(s)(t) = ts$, for $s \in S_g$ and $t \in S_{g^{-1}}$,
is an isomorphism of right $R$-modules.

\end{enumerate}
\end{prop}
\begin{proof} 
(i)$\Rightarrow$(ii):
This follows immediately from Definition~\ref{definitionepsilon}
by putting $h = e$.

(ii)$\Rightarrow$(iii):
Take $g \in G$ and $s \in S_g$. Let $\epsilon_g$ denote the multiplicative identity element
of $S_g S_{g^{-1}}$.
Using that $S$ is symmetrically graded, we may write
$s=\sum_{i=1}^n a_i b_i c_i$ where $a_1,\ldots,a_n,c_1,\ldots,c_n \in S_g$ and $b_1,\ldots,b_n \in S_{g^{-1}}$.
This yields
\begin{displaymath}
	\epsilon_g s = \sum_{i=1}^n \epsilon_g \underbrace{a_i b_i}_{\in S_g S_{g^{-1}}} c_i = \sum_{i=1}^n a_i b_i c_i = s
\end{displaymath}
and similarly
\begin{displaymath}
	 s \epsilon_{g^{-1}} = \sum_{i=1}^n a_i \underbrace{b_i c_i }_{\in S_{g^{-1}} S_g} \epsilon_{g^{-1}} = 
	\sum_{i=1}^n a_i b_i c_i = s.
\end{displaymath}

(iii)$\Rightarrow$(i):
Take $g,h \in G$. Then it follows that
$$S_g S_h = 
\epsilon_g S_g S_h \subseteq 
S_g S_{g^{-1}} S_g S_h \subseteq 
S_g S_{g^{-1}} S_{gh} \subseteq 
S_g S_{h}$$
and
$$S_g S_h = 
S_g S_h \epsilon_{ h^{-1} } \subseteq 
S_g S_h S_{h^{-1}} S_h \subseteq 
S_{gh}  S_{h^{-1}} S_h =
S_{g} S_h.$$ 
It is clear that $\epsilon_g$ is a multiplicative identity element of $S_g S_{g^{-1}}$.

(iii)$\Rightarrow$(iv):
Suppose that (iii) holds.
From the relation $\epsilon_{g^{-1}} \in S_{g^{-1}} S_g$ it follows that
there is $n \in \mathbb{N}$ and $u_i \in S_{g^{-1}}$ and
$v_i \in S_g$, for $i \in \{1,\ldots,n\}$, such that
$\sum_{i=1}^n u_i v_i = \epsilon_{g^{-1}}$.
For each $i \in \{ 1,\ldots,n \}$, define the left
$R$-linear map $f_i : S_g \rightarrow R$ by the relations
$f_i (s) = s u_i$, for $s \in S_g$. Take $s \in S_g$.
Then $s = s \epsilon_{g^{-1}} = 
\sum_{i=1}^n s u_i v_i = \sum_{i=1}^n f_i(s) v_i$.
Therefore $\{v_i\}_{i=1}^n$ and $\{f_i\}_{i=1}^n$
form ''dual bases'' and $S_g$ is therefore 
finitely generated and projective as a left $R$-module.
Next we show that $n_g$ is a monomorphism.
Suppose that $s \in S_g$ satisfies $n_g(s)=0$.
Then $s = \epsilon_g s \in S_g S_{g^{-1}} s =
S_g n_g(s)( S_{g^{-1}} ) = \{ 0 \}$.
Therefore $s = 0$.
Now we show that $n_g$ is surjective.
There is $n \in \mathbb{N}$ and $a_i \in S_g$ and
$v_i \in S_{g^{-1}}$, for $i \in \{1,\ldots,n\}$, such that
$\sum_{i=1}^n a_i b_i = \epsilon_g$.
For each $i \in \{ 1,\ldots,n \}$, define the left
$R$-linear map $g_i : S_{g^{-1}} \rightarrow R$ by the relations
$g_i (s) = s a_i$, for $s \in S_{g^{-1}}$. Take $t \in S_{g^{-1}}$.
Then $t = t \epsilon_g = 
\sum_{i=1}^n t a_i b_i = \sum_{i=1}^n g_i(t) b_i$.
Take $f \in {\rm Hom}_R ( {}_R S_{g^{-1}} , R )_R$.
Then $f(t) = \sum_{i = 1}^n g_i(t) f(b_i) = 
t \sum_{i=1}^n a_i f(b_i) = n_g( \sum_{i=1}^n a_i f(b_i) )(t)$.
Therefore, $f = n_g( \sum_{i=1}^n a_i f(b_i) )$.
Hence, $n_g$ is surjective.

(iv)$\Rightarrow$(iii):
Take $a \in S_g$ and $b \in S_{g^{-1}}$.
Suppose that the left $R$-module $S_g$ is finitely generated
and projective, and that the map 
$n_g$ is an isomorphism of right $R$-modules.
The dual basis lemma shows that there are 
$b_1,\ldots,b_n \in S_{g^{-1}}$ and
$f_1,\ldots,f_n \in {\rm Hom}_R (S_{g^{-1}} , R )$ such that
$b = \sum_{i=1}^n f_i(b) b_i$.
For every $i \in \{1,\ldots,n\}$, there is $a_i \in S_g$ such that $n_g(a_i) = f_i$.
Hence $b = \sum_{i=1}^n n_g(a_i)(b) b_i = 
\sum_{i=1}^n b a_i b_i = b \epsilon_g$, where 
$\epsilon_g = \sum_{i=1}^n a_i b_i \in S_g S_{g^{-1}}$.
This shows that $S_{g^{-1}} ( 1 - \epsilon_g ) = \{ 0 \}$.
Therefore $S_{g^{-1}} (1 - \epsilon_g) a = \{ 0 \}$ and thus
$n_g( (1-\epsilon_g) a ) (S_{g^{-1}}) = \{ 0 \}$.
This implies that $n_g ( (1-\epsilon_g)a ) = 0$.
But since $n_g$ is injective, we finally get that
$(1 - \epsilon_g)a = 0$ and hence $a = \epsilon_g a$.
Therefore $S$ is epsilon-strongly graded by $G$. 
\end{proof}


\begin{prop}
If $S$ is epsilon-strongly graded by $G$,
then $S$ is strongly graded by $G$ if and only if 
for every $g \in G$ the equality $\epsilon_g = 1$ holds.
\end{prop}

\begin{proof}
Suppose that $S$ is strongly graded by $G$.
Take $g \in G$. Since $S_g S_{g^{-1}} = R$, we get that $\epsilon_g = 1$.
Now suppose that $S$ is epsilon-strongly graded with $\epsilon_g = 1,$
for all $g \in G$. Since $S_g S_{g^{-1}}$ is a unital ideal
of $R$ with 1 as a multiplicative identity, it follows that
$R = R1 \subseteq R S_g S_{g^{-1}} \subseteq R$.
Therefore $S_g S_{g^{-1}} = R$. 
\end{proof}

\section{Separability}\label{sectionseparability}

In this section, 
we shall assume that $S$ is an arbitrary unital ring which is epsilon-strongly graded by $G$.  We will, for each $g \in G$, introduce 
an additive function $\gamma_g : S \rightarrow S$
(see Definition~\ref{definitiongamma}).
These functions will in turn be used to define a trace function 
${\rm tr}_{\gamma} : Z(R)_{\rm fin} \rightarrow Z(R)$
(see Definition~\ref{definitiontrace}).
At the end of this section, we prove Theorem~\ref{maintheorem}.
Let $\mathbb{N}$ denote the set of positive integers.

\begin{defi}\label{definitiongamma}
Let $g \in G$ be arbitrary.
From the relation $\epsilon_g \in S_g S_{g^{-1}}$ 
it follows that there is $n_g \in \mathbb{N}$,
and $u_g^{(i)} \in S_g$ and $v_{g^{-1}}^{(i)} \in S_{g^{-1}}$,
for $i \in \{1,\ldots,n_g\}$,
such that $\sum_{i=1}^{n_g} u_g^{(i)} v_{g^{-1}}^{(i)} = \epsilon_g$.
Unless otherwise stated, the elements 
$u_g^{(i)}$ and $v_{g^{-1}}^{(i)}$ are fixed.
We also assume that $n_e = 1$ and $u_e^{(1)} = v_e^{(1)} = 1$.
Define the additive function $\gamma_g : S \rightarrow S$ by
$\gamma_g(s) = \sum_{i=1}^{n_g} u_g^{(i)} s v_{g^{-1}}^{(i)}$, for $s \in S$.
\end{defi}

\begin{prop}
For any $g \in G$ and $r \in Z(R)$, the definition of $\gamma_g(r)$
does not depend on the choice of the 
elements $u_g^{(i)}$ and $v_{g^{-1}}^{(i)}$.
\end{prop}

\begin{proof}
Take $m_g \in \mathbb{N}$,
$s_g^{(j)} \in S_g$ and $t_{g^{-1}}^{(j)} \in S_{g^{-1}}$,
for $j \in \{1,\ldots,m_g\}$,
such that $\sum_{j=1}^{m_g} s_g^{(j)} t_{g^{-1}}^{(j)} = \epsilon_g$.
Then
\begin{displaymath}
\gamma_g(r) = 
\sum_{i=1}^{n_g} u_g^{(i)} r v_{g^{-1}}^{(i)} =
\sum_{i=1}^{n_g} \epsilon_g u_g^{(i)} r v_{g^{-1}}^{(i)} =
\sum_{i=1}^{n_g} \sum_{j=1}^{m_g} s_g^{(j)} t_{g^{-1}}^{(j)} u_g^{(i)} r v_{g^{-1}}^{(i)}.
\end{displaymath}
Since $t_{g^{-1}}^{(j)} u_g^{(i)} \in R$ and $r \in Z(R)$,
the last sum equals
\begin{displaymath}
\sum_{i=1}^{n_g} \sum_{j=1}^{m_g} s_g^{(j)} r t_{g^{-1}}^{(j)} u_g^{(i)} v_{g^{-1}}^{(i)} =
\sum_{j=1}^{m_g} s_g^{(j)} r t_{g^{-1}}^{(j)} \epsilon_g = 
\sum_{j=1}^{m_g} s_g^{(j)} r t_{g^{-1}}^{(j)}.
\end{displaymath}
\end{proof}

\begin{prop}\label{composition}
For any $g,h \in G$ and $r \in Z(R)$,
$\gamma_g( \gamma_h ( r ) ) = \gamma_{gh}(r) \epsilon_g$ holds.
\end{prop}

\begin{proof}  
From the definitions of $\gamma_g$ and $\gamma_h$, it follows that
\begin{displaymath}
\gamma_g ( \gamma_h ( r ) ) = 
\sum_{i=1}^{n_h} \gamma_g ( u_h^{(i)} r v_{h^{-1}}^{(i)} ) =
\sum_{i=1}^{n_h} \sum_{j=1}^{n_g} u_g^{(j)} u_h^{(i)} r v_{h^{-1}}^{(i)} v_{g^{-1}}^{(j)}.
\end{displaymath}
Since $u_g^{(j)} u_h^{(i)} \in S_{gh}$, the last sum equals
\begin{displaymath}
\sum_{i=1}^{n_h} \sum_{j=1}^{n_g} 
\epsilon_{gh} u_g^{(j)} u_h^{(i)} r v_{h^{-1}}^{(i)} v_{g^{-1}}^{(j)} =
\sum_{i=1}^{n_h} \sum_{j=1}^{n_g} \sum_{k=1}^{n_{gh}}  
u_{gh}^{(k)} v_{h^{-1} g^{-1}}^{(k)}
u_g^{(j)} u_h^{(i)} r v_{h^{-1}}^{(i)} v_{g^{-1}}^{(j)}.
\end{displaymath}
Using that $r \in Z(R)$ and 
$v_{h^{-1} g^{-1}}^{(k)} u_g^{(j)} u_h^{(i)} \in R$, the last sum equals
\begin{displaymath}
\sum_{i=1}^{n_h} \sum_{j=1}^{n_g} \sum_{k=1}^{n_{gh}}  
u_{gh}^{(k)} r v_{h^{-1} g^{-1}}^{(k)}
u_g^{(j)} u_h^{(i)} v_{h^{-1}}^{(i)} v_{g^{-1}}^{(j)} = 
\sum_{j=1}^{n_g} \sum_{k=1}^{n_{gh}}  
u_{gh}^{(k)} r v_{h^{-1} g^{-1}}^{(k)}
u_g^{(j)} \epsilon_h v_{g^{-1}}^{(j)}.
\end{displaymath}
Since $v_{h^{-1} g^{-1}}^{(k)} u_g^{(j)} \in S_{h^{-1}}$,
the last sum equals
\begin{displaymath}
\sum_{j=1}^{n_g} \sum_{k=1}^{n_{gh}}  
u_{gh}^{(k)} r v_{h^{-1} g^{-1}}^{(k)}
u_g^{(j)} v_{g^{-1}}^{(j)} = 
\sum_{k=1}^{n_{gh}}  
u_{gh}^{(k)} r v_{h^{-1} g^{-1}}^{(k)} \epsilon_g = 
\gamma_{gh}(r) \epsilon_g .
\end{displaymath}
\end{proof}

\begin{prop}\label{restriction}
Let $g\in G$ be arbitrary.
The additive function $\gamma_g : S \rightarrow S$
restricts to a surjective ring homomorphism 
$Z(R) \rightarrow Z(R) \epsilon_g$
which satisfies the relation
$\gamma_g(r) s_g = s_g r$, for all $r\in Z(R), s_g \in S_g$.
This function, in turn, restricts to a ring isomorphism
$Z(R) \epsilon_{g^{-1}} \rightarrow Z(R) \epsilon_g$.
\end{prop}

\begin{proof}
First we show that $\gamma_g(Z(R)) \subseteq Z(R) \epsilon_g$.
From the definition of $\gamma_g$ 
it follows that $\gamma_g(S) \subseteq S \epsilon_g$.
Thus, since $\epsilon_g$ is idempotent, we only need to show that $\gamma_g(Z(R)) \subseteq Z(R)$.
To this end, take $r \in Z(R)$ and $r' \in R$. Then,
since $v_{g^{-1}}^{(i)} r' \in S_{g^{-1}}$, we get that
\begin{displaymath}
\gamma_g(r) r' = 
\sum_{i=1}^{n_g} u_g^{(i)} r v_{g^{-1}}^{(i)} r' = 
\sum_{i=1}^{n_g} u_g^{(i)} r v_{g^{-1}}^{(i)} r' \epsilon_g =
\sum_{i=1}^{n_g} \sum_{j=1}^{n_g} u_g^{(i)} r v_{g^{-1}}^{(i)} r' u_g^{(j)} v_{g^{-1}}^{(j)}.
\end{displaymath}
Using that $r \in Z(R)$, $v_{g^{-1}}^{(i)} r' u_g^{(j)} \in R$
and $r' u_g^{(j)} \in S_g$, the last sum equals
\begin{displaymath}
\sum_{i=1}^{n_g} \sum_{j=1}^{n_g} 
u_g^{(i)}  v_{g^{-1}}^{(i)} r' u_g^{(j)} r v_{g^{-1}}^{(j)} =
\sum_{j=1}^{n_g} 
\epsilon_g r' u_g^{(j)} r v_{g^{-1}}^{(j)} =
\sum_{j=1}^{n_g} r' u_g^{(j)} r v_{g^{-1}}^{(j)} =
r' \gamma_g(r).
\end{displaymath}
This shows that $\gamma_g(r) \in Z(R)$.
Now we show that the restriction of $\gamma_g$ to $Z(R)$
respects multiplication.
Take $r,r' \in Z(R)$. Then
\begin{displaymath}
\gamma_g(r r') = 
\sum_{i=1}^{n_g} u_g^{(i)} r r' v_{g^{-1}}^{(i)} =
\sum_{i=1}^{n_g} \epsilon_g u_g^{(i)} r r' v_{g^{-1}}^{(i)} =
\sum_{i=1}^{n_g} \sum_{j=1}^{n_g}  
u_g^{(j)} v_{g^{-1}}^{(j)} u_g^{(i)} r r' v_{g^{-1}}^{(i)}.
\end{displaymath}
Since $r \in Z(R)$ and $v_{g^{-1}}^{(j)} u_g^{(i)} \in R$,
the last sum equals
\begin{displaymath}
\sum_{i=1}^{n_g} \sum_{j=1}^{n_g}  
u_g^{(j)} r v_{g^{-1}}^{(j)} u_g^{(i)} r' v_{g^{-1}}^{(i)} =
\sum_{j=1}^{n_g} u_g^{(j)} r v_{g^{-1}}^{(j)} 
\sum_{i=1}^{n_g} u_g^{(i)} r' v_{g^{-1}}^{(i)} =
\gamma_g(r) \gamma_g(r').
\end{displaymath}
Next, we show that the restriction
$Z(R) \rightarrow Z(R) \epsilon_g$ is 
surjective. Take $r \in Z(R)$.
From Proposition~\ref{composition}, we get that
$\gamma_g ( \gamma_{g^{-1}} ( r \epsilon_g ) ) = 
\gamma_e(r \epsilon_g) \epsilon_g = r \epsilon_g^2 = r \epsilon_g.$
For any $s_g \in S_g$, using that $v_{g^{-1}}^{(i)} s_g \in R$, we conclude that
\begin{displaymath}
\gamma_g(r) s_g =
\sum_{i=1}^{n_g} u_g^{(i)} r v_{g^{-1}}^{(i)} s_g
= \sum_{i=1}^{n_g} u_g^{(i)} v_{g^{-1}}^{(i)} s_g r
= \epsilon_g s_g r = s_g r.
\end{displaymath}

Finally, we need to show that the restriction 
$Z(R) \epsilon_{g^{-1}} \rightarrow Z(R) \epsilon_g$
is injective. Suppose that $r \in Z(R)$ is chosen so that 
$\gamma_g( r \epsilon_{g^{-1}} ) = 0$.
From Proposition~\ref{composition}, we get that
$r \epsilon_{g^{-1}} = r \epsilon_{g^{-1}}^2 = 
\gamma_e(r \epsilon_{g^{-1}} ) \epsilon_{g^{-1}} =
\gamma_{g^{-1}} ( \gamma_g ( r \epsilon_{g^{-1}} ) ) = 0$.
\end{proof}

\begin{rem}
Notice that, by Proposition~\ref{composition} and Proposition~\ref{restriction}, 
the collection of (restriction) maps $\gamma_g : Z(R) \epsilon_{g^{-1}} \rightarrow Z(R) \epsilon_g$, for $g\in G$,
yields a partial action of $G$ on $Z(R)$.
\end{rem}

\begin{defi}\label{definitiontrace}
Let $Z(R)_{\rm fin}$ denote the set of  $r \in Z(R)$ 
such that for all but finitely many $g \in G$,
the relation $r \epsilon_g = 0$ holds.
Define the {\it trace} function ${\rm tr}_{\gamma} : Z(R)_{\rm fin} \rightarrow Z(R)$
by ${\rm tr}_{\gamma}(r) = \sum_{g \in G} \gamma_g(r)$, for $r \in R$.
\end{defi}

\begin{rem}
Recall that if we for every $g \in G$ put
\begin{displaymath}
(S \otimes_R S)_g = 
\bigoplus\limits_{(g',g'')\in G \times G, \atop g'g''=g} S_{g'} \otimes_R S_{g''},
\end{displaymath}
then this defines a graded $R$-bimodule structure on $S \otimes_R S$.
We will refer to this as the $G$-grading of $S \otimes_R S$.
There is another type of grading on $S \otimes_R S$ that we will also use.
If we, for every $(g,h) \in G \times G$, put 
\begin{displaymath}
(S \otimes_R S)_{(g,h)} = S_g \otimes_R S_h,
\end{displaymath}
then this defines a graded additive structure on $S \otimes_R S$.
We will refer to this as the $(G \times G)$-grading on $S \otimes_R S$.
For more details concerning these gradings, see \cite{nas82}. 
\end{rem}

\subsection*{Proof of Theorem~\ref{maintheorem}.}
First we show the  ``if'' statement.
Suppose that there is $c \in Z(R)_{\rm fin}$ such that ${\rm tr}_{\gamma}(c) = 1$.
We wish to show that $S/R$ is separable.
To this end, put 
$x = \sum_{g \in G} \sum_{i=1}^{n_g} u_g^{(i)} c \otimes v_{g^{-1}}^{(i)}$.
From the definition of $Z(R)_{\rm fin}$ it follows that
$x$ is well-defined, since for all but finitely many $g \in G$,
we get that $u_g^{(i)} c  = u_g^{(i)} \epsilon_{g^{-1}} c = 
u_g^{(i)} 0 = 0$. Now
\begin{displaymath}
m(x) = \sum_{g \in G} \sum_{i=1}^{n_g} u_g^{(i)} c v_{g^{-1}}^{(i)} =
\sum_{g \in G} \gamma_g(c)  = {\rm tr}_{\gamma}(c) = 1.
\end{displaymath}
Next we show that $x$ commutes with all elements of $S$.
To this end, take $h \in G$ and $s \in S_h$. 
Using that $s u_g^{(i)} \in S_{hg}$, we get that
\begin{align*}s x& = 
\sum_{g \in G} \sum_{i=1}^{n_g} s u_g^{(i)} c \otimes v_{g^{-1}}^{(i)} =
\sum_{g \in G} \sum_{i=1}^{n_g} \epsilon_{hg} s u_g^{(i)} c \otimes v_{g^{-1}}^{(i)} \\
&=\sum_{g \in G} \sum_{i=1}^{n_g} \sum_{j=1}^{n_{hg}}
u_{hg}^{(j)} v_{g^{-1}h^{-1}}^{(j)} s u_g^{(i)} c \otimes v_{g^{-1}}^{(i)}.\end{align*}
Since $c \in Z(R)$ and $v_{g^{-1}h^{-1}}^{(j)} s u_g^{(i)} \in R$,
the last sum equals
\begin{displaymath}
\sum_{g \in G} \sum_{i=1}^{n_g} \sum_{j=1}^{n_{hg}}
u_{hg}^{(j)} c \otimes v_{g^{-1}h^{-1}}^{(j)} s u_g^{(i)} v_{g^{-1}}^{(i)} = 
\sum_{g \in G} \sum_{j=1}^{n_{hg}}
u_{hg}^{(j)} c \otimes v_{g^{-1}h^{-1}}^{(j)} s\epsilon_g.
\end{displaymath}
Using that $v_{g^{-1}h^{-1}}^{(j)} s \in S_{g^{-1}}$, 
the last sum equals
$\sum_{g \in G} \sum_{j=1}^{n_{hg}}
u_{hg}^{(j)} c \otimes v_{g^{-1}h^{-1}}^{(j)} s = x s.$
Therefore, $x$ is a separability
element for $S/R$.

Now we show the  ``only if'' statement.
Suppose that $x \in S \otimes_R S$ is a separability element for $S/R$.
Then $x$ satisfies
$m(x) = 1$ and, for each $s \in S$, the relation $x s = s x$ holds.
From the $G$-grading on $S \otimes_R S$ it follows
that there are unique $y \in (S \otimes_R S)_e$
and $z \in \oplus_{g \in G \setminus \{ e \}} (S \otimes_R S)_g$
such that $x = y + z$. Then we get that 
$1 = m(x) = m(y) + m(z)$. But since $1 \in R$ and 
$m(z) \in \oplus_{g \in G \setminus \{ e \}} S_g$,
we get that $m(z)=0$ and $m(y)=1$.
We claim that for each $s \in S$, the relation $y s = s y$ holds.
To this end, take $h \in G$ and $s \in S_h$.
Then $0 = s x - x s = sy - ys + sz - zs$.
Since $sy - ys \in (S \otimes_R S)_h$ and $sz - zs \in 
\sum_{g \in G \setminus \{ e \}} \left( (S \otimes_R S)_{hg} + (S \otimes_R S)_{gh} \right)$,
we get that $sy - ys = 0$ and $sz - zs = 0$.
In particular, $y$ is also a separability element for $S/R$.
Now, for each $g \in G$, there is $d_g \in S_g \otimes_R S_{g^{-1}}$
such that for all but finitely many $g \in G$, $d_g = 0$,
and $y = \sum_{g \in G} d_g$.
Furthermore, for each $g \in G$, there is $l_g \in \mathbb{N}$, and
$a_g^{(i)} \in S_g$ and $b_{g^{-1}}^{(i)} \in S_{g^{-1}}$,
for $i \in \{1,\ldots,l_g\}$, such that
$d_g = \sum_{i=1}^{l_g} a_g^{(i)} \otimes b_{g^{-1}}^{(i)}$.
For each $g \in G$, put $c_g = m(d_g)$.
Take $h \in G$. 
From the $(G \times G)$-grading on $S \otimes_R S$ it follows that
$d_h$ commutes with every $r \in R$.
Therefore, it follows that $c_h \in Z(R)$.
Take $s \in S_h$.
From the $(G \times G)$-grading on $S \otimes_R S$ it 
also follows that $s d_e = d_h s$.
Applying $m$ gives us that
$s c_e = c_h s$.
In particular, we get that
\begin{displaymath}
\gamma_h(c_e) = 
\sum_{i=1}^{n_h} u_h^{(i)} c_e v_{h^{-1}}^{(i)} =
\sum_{i=1}^{n_h} c_h u_h^{(i)} v_{h^{-1}}^{(i)} = 
c_h \epsilon_h = \epsilon_h c_h
= \epsilon_h m(d_h) = m(d_h) = c_h.
\end{displaymath}
Using that $d_h = 0$, for all but finitely many $h \in G$,
the same holds for $c_h$ and hence also for $\gamma_h(c_e)$.
By summing over $h \in G$, we get that
\begin{displaymath}
{\rm tr}_{\gamma}(c_e) = 
\sum_{h \in G} \gamma_h(c_e) = 
\sum_{h \in G} c_h = 
\sum_{h \in G} m(d_h) =
m(y) = 1.
\end{displaymath}
Moreover, by Proposition~\ref{composition}
it follows that for all but finitely many
$h \in G$, the relation 
$c_e \epsilon_{h^{-1}} = \gamma_{h^{-1}} (\gamma_h (c_e)) = 0$ holds.
Thus, $c_e \in Z(R)_{\rm fin}$. 
\qed

\begin{defi}
Let $Z(R)_{\rm fin}^{\gamma}$ denote the set of 
$r \in Z(R)_{\rm fin}$ with the property that
for every $g \in G$, the relation 
$\gamma_g( r ) = r \epsilon_g$ holds.
\end{defi}

\begin{lem}\label{center}
With the above notation, the following assertions hold:
\begin{itemize}

\item[(a)] The set $Z(R)_{\rm fin}$ is a 
(possibly non-unital) subring of $Z(R)$.
$1 \in Z(R)_{\rm fin}$ 
if and only if for all but finitely many $g \in G$, $\epsilon_g = 0$.

\item[(b)] The set $Z(R)_{\rm fin}^{\gamma}$
is a (possibly non-unital) subring of $Z(R)_{\rm fin}$.
$1 \in Z(R)_{\rm fin}^{\gamma}$ 
if and only if for all but finitely many $g \in G$, $\epsilon_g = 0$.

\item[(c)] ${\rm tr}_{\gamma} (Z(R)_{\rm fin}) \subseteq Z(R)_{\rm fin}^\gamma$
and the function 
${\rm tr}_{\gamma} : Z(R)_{\rm fin} \rightarrow Z(R)_{\rm fin}^\gamma$
is a $Z(R)_{\rm fin}^\gamma$-bimodule homomorphism.

\item[(d)] Suppose that $\epsilon_g = 0$, for all 
but finitely many $g \in G$. If $r \in Z(R)_{\rm fin}^{\gamma}$
and $r$ is invertible in $Z(R)$, then $r^{-1} \in Z(R)_{\rm fin}^{\gamma}$.
\end{itemize}
\end{lem}

\begin{proof}
(a): Take $r,r' \in Z(R)_{\rm fin}$.
Let $X(r)$ denote the finite set $\{ g \in G \mid r \epsilon_g \neq 0 \}$.
Then
$X(rr') \subseteq X(r')$
which is a finite set.
Also $X(r + r') \subseteq X(r) \cup X(r')$ which is a finite set.
Therefore $Z(R)_{\rm fin}$ is a subring of $Z(R)$.
Also $1 \in Z(R)_{\rm fin}$ precisely when 
$\epsilon_g = 1 \epsilon_g = 0$ for all but finitely many $g \in G$.

(b): Take $r,r' \in Z(R)_{\rm fin}^{\gamma}$.
From (a) we know that $rr',r+r' \in Z(R)_{\rm fin}$.
Take $g \in G$. From Proposition~\ref{restriction},
it follows that $\gamma_g(rr') = \gamma_g(r) \gamma_g(r') = 
r \epsilon_g r' \epsilon_g = rr' \epsilon_g$.
Also $\gamma_g(r + r') = \gamma_g(r) + \gamma_g(r') = 
r \epsilon_g + r' \epsilon_g = (r+r') \epsilon_g$.
Therefore $rr',r+r' \in Z(R)_{\rm fin}^{\gamma}$.
Since $\gamma_g(1) = \epsilon_g$, the last part
of (b) follows from (a).

(c): Take $r\in Z(R)_{\rm fin}$ and $g\in G$. 
By Proposition~\ref{composition},
we get that 
$\gamma_g( {\rm tr}_{\gamma}(r) ) = 
\sum_{h \in G} \gamma_g( \gamma_h(r) ) = 
\sum_{h \in G} \gamma_{gh}(r) \epsilon_{g} = 
{\rm tr}_{\gamma}(r) \epsilon_g.$
Therefore, ${\rm tr}_{\gamma}(r)\in Z(R)_{\rm fin}^\gamma.$ 
To prove the last statement, take $r \in Z(R)_{\rm fin}$  
and $r',r'' \in  Z(R)_{\rm fin}^{\gamma}.$ 
Then Proposition~\ref{restriction} implies that
\begin{align*}
{\rm tr}_{\gamma}(r' r r'') &=
\sum_{g \in G} \gamma_g (r' r r'') = 
\sum_{g \in G}\gamma_g(r') \gamma_g(r) \gamma_g(r'') =
\sum_{g \in G} r' \epsilon_g \gamma_g(r) r'' \epsilon_g  \\
&= \sum_{g \in G} r' \epsilon_g \gamma_g(r) \epsilon_g r'' =
\sum_{g \in G} r' \gamma_g(r) r'' = r' {\rm tr}_{\gamma}(r) r''.
\end{align*}

(d): From the relation $r r^{-1} = 1$, we get that
$\gamma_g(r) \gamma_g(r^{-1}) = \epsilon_g$.
Since $r \in Z(R)_{\rm fin}^{\gamma}$, we get that
$r \epsilon_g \gamma_g( r^{-1} ) = \epsilon_g$.
Thus, $r^{-1} r \gamma_g( r^{-1} ) = r^{-1} \epsilon_g$.
Hence, $\gamma_g( r^{-1} ) = r^{-1} \epsilon_g$.
\end{proof}

\begin{cor}\label{separablecorollary}
Suppose that $\ep_g=0,$ for all but finitely many $g\in G.$ 
If ${\rm tr}_{\gamma}(1)$ is invertible in $R$, then $S/R$ is separable.
\end{cor}

\proof 
From Lemma~\ref{center}(d), we get 
that ${\rm tr}_{\gamma}(1)^{-1} \in Z(R)_{\rm fin}^\gamma$. 
Thus, by Lemma~\ref{center}(c), it follows that 
${\rm tr}_{\gamma}({\rm tr}_{\gamma}(1)\m 1) = 
{\rm tr}_{\gamma}(1)\m {\rm tr}_{\gamma}(1) = 1$.
Hence, $S/R$ is separable due to Theorem~\ref{maintheorem}.
\endproof

\begin{rem}
The sufficient condition concerning invertibility of ${\rm tr}_{\gamma}(1)$
in Corollary~\ref{separablecorollary} is not necessary for separability
(see Proposition~\ref{specialcase}).
\end{rem}

\section {Semisimplicity, Hereditarity and Frobenius Properties}\label{sectionsemisimplicity}

In this section, we use Theorem~\ref{maintheorem}
to find criteria for when epsilon-strongly graded rings
are semisimple, hereditary and Frobenius
(see Theorem~\ref{theoremhereditary} and
Theorem~\ref{theoremfrobenius}).
For the rest of this section, $S/R$ denotes a ring extension.
Let ${\rm res}$ denote the restriction 
functor $S$-${\rm mod} \rightarrow R$-${\rm mod}$.
The following two results are quite well-known, but, for the 
convenience of the reader, we have chosen to include
the proofs.

\begin{prop}\label{wellknown}
Let $S/R$ be separable
and let $M$ be a left (right) $S$-module.
If ${\rm res}(M)$ is left (right) projective, 
then $M$ is left (right) projective.
\end{prop}

\begin{proof}
We only show the ``left'' part of the proof.
The ``right'' part is shown in an analogous way
and is therefore omitted.
Take $n \in \mathbb{N}$ and $s_j,t_j \in S$, for $j \in \{1,\ldots,n\}$, such that
$x = \sum_{j=1}^n s_j \otimes t_j$ is a separability element of $S \otimes_R S$.
Since ${\rm res}(M)$ is projective, $M$ has a dual $R$-basis 
$\{ m_i \}_{i \in I}$ and $\{ f_i \}_{i \in I}$.
For each $i \in I$, define $F_i : M \rightarrow S$
by $F_i(m) = \sum_{j=1}^n s_j f_i(t_j m)$, for $m \in M$.
We wish to show that $\{ F_i \}_{i \in I}$ 
and $\{ m_i \}_{i \in I}$ is a dual $S$-basis for $M$.
First of all, clearly, each $F_i$ is additive.
Next, for each $m \in M$, 
$\sum_{i \in I} F_i(m) m_i = 
\sum_{j=1}^n s_j \sum_{i \in I} f_i(t_j m) m_i =
\sum_{j=1}^n s_j t_j m = 1m = m$.
Finally, take $i \in I$, $s \in S$ and $m \in M$. 
We need to show that $F_i(sm) = sF_i(m)$, or, in other words, that
$\sum_{j=1}^n s_j f_i(t_j s m) = \sum_{j=1}^n s s_j f(t_j m)$. 
To see this, first notice that $xs = sx$ implies that 
$\sum_{j=1}^n s_j \otimes t_j s = \sum_{j=1}^n s s_j \otimes t_j$.
From the left $S \otimes_R S$-module structure on $S \otimes_R M$
it follows that 
$\sum_{j=1}^n s_j \otimes t_j s m = \sum_{j=1}^n s s_j \otimes t_j m$.
By applying the function 
$S \otimes_R M \ni a \otimes b \mapsto a \otimes f_i(b) \in S \otimes_R M$
to the last equality, we get that
$\sum_{j=1}^n s_j \otimes f_i(t_j s m) = \sum_{j=1}^n s s_j \otimes f_i(t_j m)$.
Finally, by applying the function
$S \otimes_R M \ni a \otimes b \mapsto ab \in M$ to 
the last equality, we get that
$\sum_{j=1}^n s_j f_i(t_j s m) = \sum_{j=1}^n s s_j f_i(t_j m)$.
\end{proof}

\begin{lem}\label{lemmaprojective}
$S$ is projective as a left $R$-module, if and only if,
the functor ${\rm res}$ preserves projectives.
\end{lem}

\begin{proof}
Suppose that $M$ is a projective left $S$-module.
We wish to show that $M$, considered as a left $R$-module, is projective. 
Take a dual $S$-basis $\{m_i,f_i\}_{i\in I}$ for $M$
and a dual $R$-basis $\{s_j, g_j\}_{j\in J}$ for $S$.  
Then $\{s_j m_i, g_j \circ f_i\}_{(i,j)\in I\times J}$ 
is a dual basis of $M$ as a left $R$-module. 
Indeed, for $m\in M$ we have that 
\begin{displaymath}
\sum_{(i,j)\in I\times J}(g_j \circ f_i)(m) (s_j m_i) =
\sum_i \left( \sum_j g_j ( f_i(m)) s_j \right) m_i = 
\sum_if_i(m)m_i = m.
\end{displaymath}
The converse is clear.
\end{proof}

Recall that $R$ is called left (right) semisimple if all
left (right) $R$-modules are semisimple. 
Notice that since $R$ is left semisimple if and only if $R$ is right semisimple, 
(see \cite[Corollary (3.7)]{lam1991})
the left/right distinction is therefore unnecessary.
Recall that $R$ is called left (right) hereditary
if all submodules of left (right) projective modules over $R$
are again projective. 

\begin{cor}\label{corsep}
Let $S/R$ be separable.
If $R$ is semisimple (left/right here\-ditary and 
$S$ is projective as a left/right $R$-module),
then $S$ is semisimple (left/right here\-ditary).
\end{cor}

\begin{proof}
This follows from Proposition~\ref{wellknown}, Lemma~\ref{lemmaprojective} 
and the fact that a ring is semisimple (hereditary) if and only if
every module (or submodule of a projective module)
over the ring is projective (see \cite[Theorem (2.8)]{lam1991}).
\end{proof}

It is easy to see that if $S$ is a ring graded by
a group $G$ and we put $R = S_e$, then semisimplicity
(left/right hereditarity) of $R$ is always necessary
for $S$ to be semisimple (left/right hereditary).
In fact, this is true even in the more general setting 
of rings graded by categories (see \cite[Proposition 3]{L06}).
Now we determine sufficient conditions for 
semisimplicity (left/right hereditarity).

\begin{thm}\label{theoremhereditary}
Let $S$ be epsilon-strongly graded by
a group $G$ and put $R = S_e$.
Suppose that $R$ is semisimple (hereditary). 
If $1 \in {\rm tr}_{\gamma}(Z(R)_{\rm fin})$
(and every $S_g$, for $g \in G$, is projective as a
left/right $R$-module),
then $S$ is semisimple (left/right hereditary).
In particular, if $\epsilon_g = 0$ for all but finitely many $g\in G$ and 
${\rm tr}_{\gamma}(1)$ is invertible in $R$ 
(and every $S_g$, for $g \in G$, is projective as a left/right $R$-module),
then $S$ is semisimple (left/right hereditary). 
\end{thm}

\begin{proof}
This follows from Theorem~\ref{maintheorem},
Corollary~\ref{separablecorollary}, Corollary~\ref{corsep}
and the fact that a direct sum of projective modules is projective.
\end{proof}

Recall that $S/R$ is called a {\it Frobenius extension} 
if there is a finite set $J$,
$x_j,y_j \in S$, for $j \in J$, and an
$R$-bimodule map $E : S \rightarrow R$ such that,
for every $s \in S$, the equalities
$s = \sum_{j \in J} x_j E(y_j s) = 
\sum_{j \in J}^n E(s x_j)y_j$ hold.
In that case, $(E,x_j,y_j)$ is called a
{\it Frobenius system}.

\begin{thm}\label{theoremfrobenius}
If $S$ is epsilon-strongly graded by a finite group $G$
and we put $R = S_e$, then $S/R$ is a Frobenius extension.
\end{thm}

\begin{proof} 
Put $J = \{ (g , i) \mid g \in G, \ 1 \leq i \leq n_g \},$ where $n_g$ is given by Definition~\ref{definitiongamma}.
Since $G$ is finite, $J$ is finite.
For each $j = (g,i) \in J$, define
$x_j = u_g^{(i)}$ and $y_j = v_{g^{-1}}^{(i)}$.
Define $E : S \rightarrow R$ by 
$E (s) = s_e$, for $s \in S$.
Then, clearly, $E$ is an $R$-bimodule map.
Take $s \in S$. Then
\begin{align*}\sum_{j \in J} x_j E(y_j s)& =
\sum_{g \in G} \sum_{i=1}^{n_g} u_g^{(i)} E( v_{g^{-1}}^{(i)} s ) =
\sum_{g \in G} \sum_{i=1}^{n_g} u_g^{(i)} v_{g^{-1}}^{(i)} s_g
 =
\sum_{g \in G} \epsilon_g s_g = s\end{align*} and
\begin{align*}\sum_{j \in J}^n E(s x_j)y_j =
\sum_{g \in G} \sum_{i=1}^{n_g} E( s u_g^{(i)} ) v_{g^{-1}}^{(i)} = 
\sum_{g \in G} \sum_{i=1}^{n_g} s_{g^{-1}} u_g^{(i)} v_{g^{-1}}^{(i)} =
\sum_{g \in G} s_{g^{-1}} \epsilon_{g} 
= s.\end{align*}
\end{proof}

\begin{rem}
The conclusion of Theorem~\ref{theoremfrobenius} follows from 
Proposition~\ref{epsilon1}(iv).
Indeed, since $G$ is finite, it follows that 
$S$ is finitely generated and projective as a left $R$-module.
Using the notation used in Proposition~\ref{epsilon1}, put $n = \oplus_{g \in G} n_g$.
Then $n$ is an isomorphism of $R$-modules
$S_R \rightarrow {\rm Hom}_R( {}_R S , R )_R$.
Hence $S/R$ is a Frobenius extension, according to \cite[Theorem 1.2]{kadison1999}. 
\end{rem}

Recall that if $T$ is a non-empty subset of $S$, 
then $C_S(T)$ denotes the set of $s \in S$
such that for every $t \in T$, the relation $st=ts$ holds.

\begin{prop}\label{propkadison}
Let $S/R$ be a Frobenius extension with Frobenius system
$(E,x_j,y_j)$. Then $S/R$ is separable if and only if there is
$d \in C_S(R)$ such that $\sum_{j \in J} x_j d y_j = 1$.
\end{prop}

\begin{proof}
See \cite[Corollary 2.17]{kadison1999}.
\end{proof}

\begin{rem}
Suppose that $S$ is epsilon-strongly graded by a finite group $G$ and put $R = S_e$.
Using Theorem~\ref{theoremfrobenius} and Proposition~\ref{propkadison},
we can, in this case, prove Theorem~\ref{maintheorem} in a different way.
Indeed, using the above results, we can conclude that 
$S/R$ is separable if and only if there is $d \in C_S(R)$ such that 
$\sum_{g \in G} \gamma_g(d) = 
\sum_{g \in G} \sum_{i=1}^{n_g} u_g^{(i)} d v_{g^{-1}}^{(i)}= 1$.
Since $1 \in R$ it follows from the grading that $S/R$ is
separable if and only if there is $c=d_e \in C_R(R) = Z(R)$ such that
${\rm tr}_{\gamma}(c) = 1$. 
\end{rem}

\section{Simplicity}\label{simplicity}

In this short section, we show that a result concerning simplicity
for strongly graded rings from \cite[Theorem 6.6]{O09}
can be generalized to epsilon-strongly graded rings
(see Proposition~\ref{Prop:Simplicity}).
Throughout this section, 
$S$ denotes an arbitrary unital ring which is epsilon-strongly graded by $G$ and we put $R = S_e$.
Recall that $R$ is called a {\it maximal
commutative subring} of $S$ if $C_S(R) = R$.

\begin{lem}\label{Lem:NonDeg}
If $I$ is a non-zero ideal of $S$, then $I \cap C_S(Z(R)) \neq \{0\}$.
\end{lem}

\begin{proof}
We claim that $S$ is right non-degenerate in the sense of \cite[Definition 2]{OL12}.
If we assume that the claim holds, then the desired result follows from \cite[Theorem 3]{OL12}.
Now we show the claim.
Take $g\in G$ and any non-zero $s\in S_g$.
Seeking a contradiction,
suppose that $s S_{g^{-1}} = \{ 0 \}$.
Then $sS_{g^{-1}}S_g = \{ 0 \}$. But since $\epsilon_{g^{-1}} \in S_{g^{-1}}S_g$ we get that $s=s\epsilon_{g^{-1}} = 0$, which is a contradiction.
\end{proof}

Recall that an ideal $I$ of $S$ is said to be \emph{graded} 
if $I = \oplus_{g\in G} (I \cap S_g)$ holds.
If $\{0\}$ and $S$ are the only graded ideals of $S$, 
then $S$ is said to be \emph{graded simple}.

\begin{prop}\label{Prop:Simplicity}
If $R$ is a maximal commutative subring of $S$, 
then $S$ is simple if and only if $S$ is graded simple.
\end{prop}

\begin{proof}
The ``only if'' statement is clear.
Now we show the ``if'' statement.
Let $I$ be a non-zero ideal of $S$.
By the assumption we have $C_S(Z(R))=C_S(R)=R$.
Hence, by Lemma~\ref{Lem:NonDeg} the set $J= I \cap C_S(Z(R))$ is a non-zero ideal of $R$.
The set $SJS$ is a non-zero graded ideal of $S$ and thus, by graded simplicity of $S$, we get that $S=SJS=J$.
This shows that $S$ is a simple ring.
\end{proof}

\section{Partial Crossed Products}\label{sectionpartialcrossedproducts}

In this section, we introduce epsilon-crossed products
(see Definition~\ref{epsiloncrossedproduct}).
We show that the class of epsilon-crossed products
coincides with the class of unital partial crossed products
(see Theorem~\ref{correspondence}).
Thereby, we obtain, in the special case of unital 
partial crossed products, a short proof of
a more  general result by Dokuchaev, Exel and Sim\'{o}n \cite[Theorem 6.1]{dokuchaev2008}
concerning when graded rings can be presented as partial crossed products.
At the end of this section, we use Theorem~\ref{maintheorem}
to reformulate Theorem~\ref{separabilitypartial} so that
it holds for any, possibly infinite, group $G$
(see Theorem~\ref{genseparabilitypartial}).

\begin{defi}
Let $S$ be a ring which is epsilon-strongly graded by $G$.
Take $g \in G$ and $s \in S_g$. Then $s$ is called 
\emph{epsilon-invertible} if there is $t \in S_{g^{-1}}$
such that $st = \epsilon_g$ and $ts = \epsilon_{g^{-1}}$.
We will refer to $t$ as the \emph{epsilon-inverse} of $s$.
\end{defi}

The usage of the term ``the epsilon-inverse'' 
is justified by the next result.

\begin{prop}
Epsilon-inverses are unique.
\end{prop}

\begin{proof}
Suppose that $g \in G$, $s \in S_g$ and $r,t \in S_{g^{-1}}$
satisfy the equalities
$st = sr = \epsilon_g$ and $ts = rs = \epsilon_{g^{-1}}$.
Then $r = r \epsilon_g = r s t = \epsilon_{g^{-1}} t = t$.
\end{proof}

\begin{defi}\label{epsiloncrossedproduct}
Let $S$ be a ring which is epsilon-strongly graded by $G$.
We say that $S$ is an \emph{epsilon-crossed product by $G$} if for 
each $g \in G$, there is an epsilon-invertible 
element in $S_g$.
\end{defi}

\begin{thm}\label{correspondence}
Let $S$ be a ring which is epsilon-strongly graded by $G$.
Then $S$ is an epsilon-crossed product if and only if 
$S$ is a unital partial crossed product.
\end{thm}

\begin{proof}
First we show the ``only if'' statement.
Suppose that $S$ is an epsilon-crossed product.
We will present $S$ as a unital partial crossed product.
Take $g,h \in G$.
Fix an epsilon-invertible 
element $s_g \in S_g$ with epsilon-inverse $t_{g^{-1}} \in S_{g^{-1}}$.
We may assume that $s_e = t_e = 1$.
Put $D_g = S_g S_{g^{-1}} = R \epsilon_g$, 
$1_g = \epsilon_g$ and $\delta_g = s_g$.
Furthermore, define 
$\alpha_g : D_{g^{-1}} \rightarrow D_g$
by $\alpha_g(r \epsilon_{g^{-1}}) = s_g r t_{g^{-1}}$, for $r \in R$. 
Then $\alpha_g$ is well-defined.
Indeed, if $r,r' \in R$ satisfy 
$r \epsilon_{g^{-1}} = r' \epsilon_{g^{-1}}$, then
$\alpha_g(r \epsilon_{g^{-1}}) = 
s_g r t_{g^{-1}} =
s_g r \epsilon_{g^{-1}} t_{g^{-1}} = 
s_g r' \epsilon_{g^{-1}} t_{g^{-1}} =
s_g r' t_{g^{-1}} =
\alpha_g(r' \epsilon_{g^{-1}})$.
The function $\alpha_g$ is bijective with inverse given by 
$\alpha_g^{-1}(r \epsilon_g) = 
t_{g^{-1}} r s_g$.
Indeed, take $r \in R$. Then
\begin{displaymath}
\alpha_g^{-1} ( \alpha_g (r\epsilon_{g^{-1}}) ) = 
t_{g^{-1}} s_g r t_{g^{-1}} s_g = 
\epsilon_{g^{-1}} r \epsilon_{g^{-1}} = 
r \epsilon_{g^{-1}}
\end{displaymath}
and
\begin{displaymath}
\alpha_g ( \alpha_g^{-1} (r\epsilon_g) ) = 
s_g t_{g^{-1}} r s_g t_{g^{-1}} = 
\epsilon_{g} r \epsilon_{g} = 
r \epsilon_{g}.
\end{displaymath}
The function $\alpha_g$ is clearly additive. Also 
$\alpha_g(\epsilon_{g^{-1}}) = 
s_g \epsilon_{g^{-1}} t_{g^{-1}} =
s_g t_{g^{-1}} = \epsilon_g$.
Now we show that $\alpha_g$ is multiplicative. Take $r,r' \in R$. Then
\begin{displaymath}
\alpha_g(r r' \epsilon_g) = 
s_g r r' t_{g^{-1}} =
s_g r \epsilon_{g^{-1}} r' t_{g¨{-1}} =
s_g r t_{g^{-1}} s_g r' t_{g^{-1}} =
\alpha_g(r \epsilon_{g^{-1}}) \alpha_g(r' \epsilon_{g^{-1}}).
\end{displaymath}
Next put $w_{g,h} = s_g s_h t_{(gh)^{-1}}$.
Since $w_{g,h} \in R$,
$\epsilon_g s_g = s_g$ and 
$t_{(gh)^{-1}} \epsilon_{gh} = t_{(gh)^{-1}}$,
it follows that $w_{g,h} \in D_g D_{gh}$.
Now we show that $w_{g,h}$ is a unit in $D_g D_{gh}$.
To this end, first notice that 
$\alpha_g( \epsilon_{g^{-1}} \epsilon_h ) = \epsilon_g \epsilon_{gh}$.
In fact, from (P2) (see below), we get that 
there is $r \in D_{g^{-1}} D_h $ such that
$\alpha_g (r) = \epsilon_g \epsilon_{gh}$.
Since $\epsilon_g \epsilon_{gh}$ is the identity
of $D_g D_{gh}$, we get that
$\alpha_g( \epsilon_{g^{-1}} \epsilon_h ) = 
\alpha_g( \epsilon_{g^{-1}} \epsilon_h ) \epsilon_g \epsilon_{gh} =
\alpha_g( \epsilon_{g^{-1}} \epsilon_h ) \alpha_g(r) =
\alpha_g( \epsilon_{g^{-1}} \epsilon_h r) =
\alpha_g(r) = \epsilon_g \epsilon_{gh}$.
Put $v_{g,h} = s_{gh} t_{h^{-1}} t_{g^{-1}}\epsilon_g \epsilon_{gh}$.
Then $v_{g,h}\in D_gD_{gh}$
 and 
\begin{align*}
w_{g,h} v_{g,h} &=  s_g s_h t_{(gh)^{-1}} s_{gh} t_{h^{-1}} t_{g^{-1}} =
s_g s_h \epsilon_{(gh)^{-1}} t_{h^{-1}} t_{g^{-1}} =
s_g s_h t_{h^{-1}} t_{g^{-1}}\\
& =  s_g \epsilon_h t_{g^{-1}} = s_g \epsilon_h \epsilon_{g^{-1}} t_{g^{-1}} =
 \alpha_g( \epsilon_h \epsilon_{g^{-1}} )=
 \epsilon_g \epsilon_{gh}\end{align*}
and
\begin{align*}v_{g,h} w_{g,h}& = 
s_{gh} t_{h^{-1}} t_{g^{-1}} s_g s_h t_{(gh)^{-1}} =
s_{gh} t_{h^{-1}} \epsilon_{g^{-1}} s_h t_{(gh)^{-1}} =
s_{gh} t_{h^{-1}} s_h t_{(gh)^{-1}} \\
&= s_{gh} \epsilon_{h^{-1}} t_{(gh)^{-1}} =
s_{gh} \epsilon_{h^{-1}} \epsilon_{(gh)^{-1}} t_{(gh)^{-1}} =  
\alpha_{gh}( \epsilon_{(gh)^{-1}} \epsilon_{h^{-1}} ) 
= \epsilon_{gh} \epsilon_g \epsilon_{gh} \\
&= \epsilon_{gh} \epsilon_g.\end{align*}
Now we check conditions (P1)-(P6) from the introduction.

(P1): Using that $\epsilon_e = 1$, we get that $D_e = R$.
Since $\gamma_e = {\rm id}_R$, we get that $\alpha_e = {\rm id}_R$.

(P2): First notice that 
$\alpha_g( D_{g^{-1}} D_h ) =
s_g D_{g^{-1}} D_h t_{g^{-1}} = 
s_g S_{g^{-1}} S_g S_h S_{h^{-1}} t_{g^{-1}}.$
Since $s_g \in S_g$, and thereby $s_g S_{g^{-1}} \in R$, 
we can conclude that 
\begin{align*}\alpha_g( D_{g^{-1}} D_h ) &=
\epsilon_g s_g S_{g^{-1}} \epsilon_{gh} S_g S_h S_{h^{-1}} t_{g^{-1}} =
\epsilon_g \epsilon_{gh} (s_g S_{g^{-1}} S_g S_h S_{h^{-1}} t_{g^{-1}})\\
& \subseteq D_g D_{gh}R= D_g D_{gh}.\end{align*}
Now we show the reversed inclusion.
Take $r \in R$.
Put $r' = t_{g^{-1}} r \epsilon_g \epsilon_{gh} s_g\in R.$
Then $\epsilon_{g^{-1}} r' = r'$.
Also, since  $\epsilon_{gh} s_g \in S_{gh} S_{(gh)^{-1}} s_g \subseteq S_{gh} S_{h^{-1}}$, it follows that $r' \epsilon_h = r'$.
Thus,
$r' \in D_{g^{-1}} D_h$.
Now,
\begin{displaymath}
\alpha_g(r') = s_g t_{g^{-1}} r \epsilon_g \epsilon_{gh} s_g t_{g^{-1}} =
\epsilon_g r \epsilon_g \epsilon_{gh} \epsilon_g =
r \epsilon_g \epsilon_{gh}.
\end{displaymath}

(P3): Take $r \in D_{h^{-1}} D_{(gh)^{-1}}$. Then 
\begin{displaymath}
\alpha_g( \alpha_h (r) ) w_{g,h} = 
s_g s_h r t_{h^{-1}} t_{g^{-1}} s_g s_h t_{(gh)^{-1}} = 
s_g s_h r t_{h^{-1}} \epsilon_{g^{-1}} s_h t_{(gh)^{-1}},
\end{displaymath}
and the last
expression equals
$$
s_g s_h r t_{h^{-1}} s_h t_{(gh)^{-1}} =
s_g (s_h r )\epsilon_{h^{-1}} t_{(gh)^{-1}} =
s_g s_h r  t_{(gh)^{-1}}  =
s_g s_h r \epsilon_{(gh)^{-1}} t_{(gh)^{-1}}  $$
$$= (s_g s_h  t_{(gh)^{-1}} ) s_{gh}r  t_{(gh)^{-1}} =
w_{g,h} s_{gh} r t_{(gh)^{-1}}  = w_{g,h} \alpha_{gh}(r).
$$

(P4): Using that $s_e = 1$, we get that
$w_{g,e} = s_g s_e t_{g^{-1}} = s_g t_{g^{-1}} = \epsilon_g$ and
$w_{e,g} = s_e s_g t_{g^{-1}} = s_g t_{g^{-1}} = \epsilon_g$.

(P6): Notice first that $S_g = D_g \delta_g$.
In fact, since $s_g \in S_g$ and $D_g \subseteq R$
it follows that $S_g \supseteq D_g s_g$.
On the other hand, take $s_g' \in S_g$.
Then $s_g' = s_g' \epsilon_{g^{-1}} =
s_g' t_{g^{-1}} s_g = 
s_g' t_{g^{-1}} \epsilon_g s_g \in D_g s_g$.
Thus $S_g \subseteq D_g s_g$.
So we get that $S = \oplus_{g \in G} D_g s_g$.
Also, if $s = \sum_{g \in G} r_g s_g$, for $r_g \in D_g$,
then the $r_g$'s are unique. Indeed, suppose that 
$r_g s_g = r_g' s_g$ for some $r_g,r_g' \in D_g$.
Using that $\epsilon_g$ is the multiplicative identity element of $D_g = S_g S_{g^{-1}}$, we get that
$r_g = 
r_g \epsilon_g =
r_g s_g t_{g^{-1}} =
r_g' s_g t_{g^{-1}} =
r_g' \epsilon_g = r_g'$.
Take $r \in D_g$ and $r' \in D_h$. Then
\begin{align*}( r s_g ) ( r' s_h ) &= r s_g \epsilon_{g^{-1}} r' \epsilon_{g^{-1}} s_h =
r s_g \epsilon_{g^{-1}} r' t_{g^{-1}} s_g s_h =
r \alpha_g( \epsilon_{g^{-1}} r' ) s_g s_h  \\
&= r \alpha_g( \epsilon_{g^{-1}} r' ) s_g s_h \epsilon_{(gh)^{-1}} =  
r \alpha_g( \epsilon_{g^{-1}} r' ) s_g s_h 
t_{(gh)^{-1}} s_{gh}\\
& = 
r \alpha_g( \epsilon_{g^{-1}} r' ) w_{g,h} s_{gh}.\end{align*}

(P5): Take $r \in D_{g^{-1}} D_h D_{hl}$. Then
\begin{displaymath}
(s_g r s_h) s_l = (\alpha_g(r) s_g s_h) s_l = 
(\alpha_g(r) w_{g,h} s_{gh} ) s_l = 
\alpha_g(r) w_{g,h} w_{gh,l} s_{ghl}
\end{displaymath}
and
\begin{displaymath}
s_g ( r s_h s_l ) = 
s_g ( r w_{h,l} s_{hl} ) =
\alpha_g ( r w_{h,l} ) w_{g,hl} s_{ghl}.
\end{displaymath}
The claim now follows from the proof of (P6) and associativity.

Now we show the ``if'' statement.
Suppose that $S = \oplus_{g \in G} D_g \delta_g$
is a unital partial crossed product. Take $g \in G$.
Since $S_g S_{g^{-1}} = D_g \delta_e$, we can put 
$\epsilon_g = 1_g \delta_e$.
What remains to show is associativity of $S$.
This has already been shown in a more general context
(see \cite[Theorem 2.4]{dokuchaev2008}).
Here we provide a short direct proof for unital twisted partial actions.
To this end, take $g,h,l \in G$, $a \in D_g$, 
$b \in D_h$ and $c \in D_l$. Then
$$
( a \delta_g b \delta_h ) c \delta_l = 
( a \alpha_g (1_{g^{-1}} b) w_{g,h} \delta_{gh} ) c \delta_l
= a \alpha_g (1_{g^{-1}} b) w_{g,h} \alpha_{gh} (1_{(gh)^{-1}} c )
 w_{gh,l} \delta_{ghl}. 
$$
By (P2), the last expression equals
$$a \alpha_g (1_{g^{-1}} b) w_{g,h} \alpha_{gh} (1_{h^{-1}} 1_{(gh)^{-1}} c )
 w_{gh,l} \delta_{ghl},
$$
which, in turn, by (P3), equals
\begin{align*}
& a \alpha_g (1_{g^{-1}} b) \alpha_g( \alpha_h( 1_{h^{-1}} 1_{(gh)^{-1}} c ) ) w_{g,h}
 w_{gh,l} \delta_{ghl} = \\ 
& a \alpha_g (1_{g^{-1}} b \alpha_h( 1_{h^{-1}} 1_{(gh)^{-1}} c ) ) w_{g,h}
 w_{gh,l} \delta_{ghl} = a \alpha_g (1_{g^{-1}} b \alpha_h( 1_{h^{-1}} c ) ) w_{g,h}
 w_{gh,l} \delta_{ghl}.
\end{align*}
By (P5), this equals
\begin{displaymath}
a \alpha_g( 1_{g^{-1}} b \alpha_h ( 1_{h^{-1}} ) w_{h,l} ) w_{g,hl} \delta_{ghl}=
a \delta_g ( b \alpha_h ( 1_{h^{-1}} c ) w_{h,l} \delta_{hl} ) = 
a \delta_g ( b \delta_h c \delta_l ).
\end{displaymath}
\end{proof}

\begin{defi}
Let $S=R\star_\alpha^w G$ be a unital partial crossed product, and
let $Z(R)_{\alpha,{\rm fin}}$ denote the set of 
$r \in Z(R)$ with the property that 
for all but finitely many $g \in G$, the relation 
$r 1_g = 0$ holds. Define the trace map 
$t_{\alpha} : Z(R)_{\alpha,{\rm fin}} \rightarrow Z(R)$
by $t_{\alpha}(r) = \sum_{g \in G} \alpha_g(r 1_{g^{-1}})$,
for $r \in Z(R)_{\alpha,{\rm fin}}$.
\end{defi}

\begin{thm}\label{genseparabilitypartial}
If $S$ is a unital partial crossed product of a group $G$ over $R$,
then $S/R$ is separable if and only if 
$1 \in {\rm tr}_{\alpha} (Z(R)_{\rm fin})$.
\end{thm}

\begin{proof}
This follows from Theorem~\ref{maintheorem} and
Theorem~\ref{correspondence}.
\end{proof}

\begin{lem}\label{lemmaproj}
Let $S$ be a unital partial crossed product 
of a group $G$ over $R$ and take $g \in G$.
If $D_g$ is projective as a left (right) 
$R$-module, then $S_g=D_g\delta_g$ is projective as a left (right)
$R$-module.
\end{lem}

\begin{proof}
The  ``left'' part is trivial since the 
left action of $R$ on $S_g$ is defined by 
the left action of $R$ on $D_g$.
Now we show the  ``right'' part.
Suppose that $D_g$ is projective as a right $R$-module.
Let $\{ d_i , f_i \}_{i \in I}$ be a dual basis
for $D_g$ as a right $R$-module.
For each $i \in I$, define $\overline{f}_i : D_g \rightarrow R$
by the relations $\overline{f}_i(d) = \alpha_g^{-1}( 1_g f_i(d) )$,
for $d \in D_g$. Then $\{ d_i \delta_g , \overline{f}_i \}_{i \in I}$
is a dual basis for $S_g = D_g \delta_g$ as a right $R$-module.
In fact, if $d \in D_g$, then
\begin{displaymath}
\sum_{i \in I} (d_i \delta_g) (\overline{f}_i(d) \delta_e) =
\sum_{i \in I} d_i \alpha_g( \overline{f}_i(d) ) w_{g,e} \delta_g =
\sum_{i \in I} d_i 1_g f_i(d) \delta_g = d \delta_g.
\end{displaymath}
\end{proof}

\begin{thm}
Let $S$ be a unital partial crossed product 
of a group $G$ over $R$,
and let $R$ is semisimple (left/right hereditary). 
If $1 \in {\rm tr}_{\gamma}(Z(R)_{\rm fin})$,
then $S$ is semisimple (left/right hereditary).
In particular, if $\epsilon_g = 0$ for all but finitely many $g\in G$ and 
${\rm tr}_{\gamma}(1)$ is invertible in $R$, 
then $S$ is semisimple (left/right hereditary). 
\end{thm}

\begin{proof}
The ``semisimple'' part follows from Theorem~\ref{theoremhereditary}.
The ``hereditary'' part follows from Theorem~\ref{theoremhereditary},
Lemma~\ref{lemmaproj} and the fact that all the ideals $D_g$, for $g \in G$,
of the hereditary ring $R$, are left/right projective.
\end{proof}

\section{Examples: A Dade-Like Construction}\label{sectionexample}

In this section, we provide a class of examples of separable 
epsilon-strongly graded rings, neither of which are strongly graded, 
nor partial crossed products, in any natural way.
Our inspiration comes from the first known 
example (due to E. Dade, according to \cite[Example 2.9]{dascalescu1999})
of a strongly graded ring which is not a crossed product.
Namely, suppose that $A$ is a commutative unital ring with
a non-zero multiplicative identity $1_A$. Put
\begin{displaymath}
S =  M_3(A),
\quad
R =  
\left( 
\begin{matrix}
  A & A & 0 \\
  A & A & 0 \\
  0 & 0 & A
 \end{matrix}
\right)
\,\, \text{  and  } \,\,
T =  
\left( 
\begin{matrix}
  0 & 0 & A \\
  0 & 0 & A \\
  A & A & 0
 \end{matrix}
\right).
\end{displaymath}
Then $S$ is strongly ${\Bbb Z}_2$-graded with $S_0 = R$ and $S_1 = T$, 
but $S$ is not a crossed product of ${\Bbb Z}_2$ over $R$ since $T$ 
does not contain any element which is invertible in $S$.
Our idea is to postulate another unital commutative ring
$B$ with a non-zero multiplicative identity $1_B$
such that $B$ is an ideal of $A$ with $B\subsetneq A$.
Now we modify Dade's example by putting
\begin{displaymath}
S =  
\left( 
\begin{matrix}
  A & A & B \\
  A & A & B \\
  B & B & A
 \end{matrix}
\right),
\quad
R =  
\left( 
\begin{matrix}
  A & A & 0 \\
  A & A & 0 \\
  0 & 0 & A
 \end{matrix}
\right)
\,\, \text{  and  } \,\,
T =  
\left( 
\begin{matrix}
  0 & 0 & B \\
  0 & 0 & B \\
  B & B & 0
 \end{matrix}
\right).
\end{displaymath}

\begin{prop}\label{propexample}
The ring $S$ is epsilon-strongly $\mathbb{Z}_2$-graded with 
$S_0 = R$ and $S_1 = T$. With this grading,
$S$ is neither strongly graded, nor
a partial crossed product.
Moreover, the ring extension $S/R$ is separable.
\end{prop}

\begin{proof}
If we put 
\begin{displaymath}
\epsilon_0 = 
\left( 
\begin{matrix}
  1_A & 0   & 0 \\
  0   & 1_A & 0 \\
  0   & 0 & 1_A
 \end{matrix}
\right)
\quad
{\rm and}
\quad
\epsilon_1 = 
\left( 
\begin{matrix}
  1_B & 0   & 0 \\
  0   & 1_B & 0 \\
  0   & 0 & 1_B
 \end{matrix}
\right),
\end{displaymath}
then it is clear that $RR = R \epsilon_0$ and
$TT = \epsilon_1 R$.
Hence $S$ is epsilon-strongly $\mathbb{Z}_2$-graded.
The equality $TT = \epsilon_1 R$ also shows that 
$S$ is not strongly graded, since $1_A \notin B$.
Seeking a contradiction, suppose that this grading presents
$S$ as a partial crossed product of $\mathbb{Z}_2$ over $R$.
Then there is $\delta_1 \in T$ and a non-zero unital ideal $D$ of $R$
such that $T = D \delta_1$.
Take unital ideals $I$ and $J$ of $A$, of which at least one is non-zero, such that
\begin{displaymath}
D =  
\left( 
\begin{matrix}
  I & I & 0 \\
  I & I & 0 \\
  0 & 0 & J
 \end{matrix}
\right),
\end{displaymath}
take $b_1,b_2,b_3,b_4 \in B$ such that
\begin{displaymath}
\delta_1 = 
\left( 
\begin{matrix}
  0   & 0   & b_3 \\
  0   & 0   & b_4 \\
  b_1 & b_2 & 0
 \end{matrix}
\right)
\end{displaymath}
and take $w_{11},w_{12},w_{21},w_{22} \in I$ and $w_{33} \in J$ such that
\begin{displaymath}
w_{1,1}^{-1} =  
\left( 
\begin{matrix}
  w_{11} & w_{12} & 0 \\
  w_{21} & w_{22} & 0 \\
  0 & 0 & w_{33}
 \end{matrix}
\right).
\end{displaymath}
From (P6) it follows that 
\begin{equation}\label{wdelta}
(w_{1,1}^{-1} \delta_1) (1_D \delta_1) = 1_D \delta_0.
\end{equation}
By a straightforward calculation, \eqref{wdelta} can be rewritten as
{\scriptsize
\begin{equation}\label{complicated}
\left( 
\begin{matrix}
  1_J b_1 ( w_{11} b_3 + w_{12} b_4 ) & 1_J b_2 ( w_{11} b_3 + w_{12} b_4 ) & 0 \\
  1_J b_1 ( w_{21} b_3 + w_{22} b_4 ) & 1_J b_2 ( w_{21} b_3 + w_{22} b_4 ) & 0 \\
  0 & 0 & 1_I w_{33} ( b_1 b_3 + b_2 b_4 )
 \end{matrix}
\right)
=
\left( 
\begin{matrix}
  1_I & 0   & 0 \\
  0   & 1_I & 0 \\
  0   & 0   & 1_J
 \end{matrix}
\right).
\end{equation}}

\noindent From \eqref{complicated} it follows in particular that
$$1_J b_1 ( w_{11} b_3 + w_{12} b_4 ) = 1_I \quad {\rm and} \quad
1_I w_{33} ( b_1 b_3 + b_2 b_4 ) = 1_J.$$
Thus, $1_I \in J$ and $1_J \in I$.
Hence $I = J$ and so we get that $1_I = 1_J \neq 0$.
By a straightforward calculation the determinant of the
left hand side of \eqref{complicated} is zero.
This contradicts the fact that the determinant of the right
hand side of \eqref{complicated} equals $1_I \neq 0$.

Now we show that $S/R$ is separable.
First of all, it is easy to show that
\begin{displaymath}
Z(R) = \left\lbrace
\left( 
\begin{matrix}
  a & 0 & 0 \\
  0 & a & 0 \\
  0 & 0 & a'
 \end{matrix}
\right)
\
\Big\lvert \ a,a' \in A
\right\rbrace.
\end{displaymath}
We know that $\gamma_0 : Z(R) \rightarrow Z(R)$ is the identity map on $Z(R)$.
Now we determine $\gamma_1 : Z(R) \rightarrow Z(R)$.
To this end, let $e_{ij}$ denote the 3$\times$3 matrix over $A$
with $1_A$ in the $ij$th position, and zeros elsewhere.
Since
\begin{displaymath}
1_B e_{13} 1_B e_{31} + 1_B e_{23} 1_B e_{32} + 1_B e_{31} 1_B e_{13} = \epsilon_1
\end{displaymath}
the map $\gamma_1 : Z(R) \rightarrow Z(R)$ is defined by
\begin{displaymath}
Z(R) \ni r \mapsto 1_B e_{13} r 1_B e_{31} + 1_B  e_{23} r 1_B e_{32} 
+ 1_B e_{31} r 1_B e_{13}.
\end{displaymath}
Thus, the trace map ${\rm tr}_{\gamma} : Z(R) \rightarrow Z(R)$ is defined by
\begin{displaymath}
Z(R) \ni r \mapsto r + 1_B e_{13} r 1_B e_{31} + 1_B  e_{23} r 1_B e_{32} 
+ 1_B e_{31} r 1_B e_{13}.
\end{displaymath}
By Theorem~\ref{maintheorem}, we can deduce that $S/R$ is separable
if we can find
\begin{displaymath}
r = 
\left(
\begin{matrix}
  a & 0 & 0 \\
  0 & a & 0 \\
  0 & 0 & a'
 \end{matrix}
 \right)
\in Z(R)
\end{displaymath}
such that 
\begin{displaymath}
{\rm tr}_{\gamma}(r) = 
\left(
\begin{matrix}
  1_A & 0 & 0 \\
  0 & 1_A & 0 \\
  0 & 0 & 1_A
 \end{matrix}
\right).
\end{displaymath}
By a straightforward calculation, the last relation is equivalent
to the set of equations
$a + 1_B a' = 1_A$ and $a' + 1_B a = 1_A.$
It is easy to see that this set of equations is satisfied 
if we e.g. put $a = 1_A$ and $a' = 1_A - 1_B$.
Therefore, $S/R$ is separable.
\end{proof}

It is easy to give concrete examples of rings
$A$ and $B$ which fit into the above construction.
In fact, from now on in this section, suppose that $\mathbb{ F}$ is a field, 
$A = \mathbb{ F} \times \mathbb{ F}$ and $B =\mathbb{ F} \times \{ 0 \}$.
In that case, the sufficient condition
for separability in Corollary~\ref{separablecorollary}
is not necessary.

\begin{prop}\label{specialcase}
With the above notation, ${\rm tr}_{\gamma}(1_R)$
is invertible in $R$ if and only if ${\rm char}(\mathbb{ F}) \neq 2$.
\end{prop}

\begin{proof}
Since 
\begin{displaymath}
1_R = 
\left(
\begin{matrix}
  1_A & 0 & 0 \\
  0 & 1_A & 0 \\
  0 & 0 & 1_A
 \end{matrix}
 \right),
\end{displaymath}
we get, from the proof of Proposition~\ref{propexample},
that 
\begin{displaymath}
{\rm tr}_{\gamma} (1_R) = 
\left(
\begin{matrix}
  1_A + 1_B & 0         & 0 \\
  0         & 1_A + 1_B & 0 \\
  0         & 0         & 1_A + 1_B
 \end{matrix}
 \right).
\end{displaymath}
Since $1_A + 1_B = (1,1) + (1,0) = (2,1)$,
we get that ${\rm tr}_{\gamma}(1_R)$
is invertible in $R$ if and only if ${\rm char}(\mathbb{F}) \neq 2$.
\end{proof}
\begin{rem}\label{rem:Bruynbook} In \cite[Remark II.5.1.6]{Bruyn1988} the authors write that 
''If $S$ is an arbitrary graded ring by a finite group $G$
we do not know whether separability of $S$ over $S_e$ implies that
$S$ is strongly graded. This seems very likely however.''
Proposition~\ref{propexample} is a counterexample to this assumption.
Moreover, it is the first known example for which all the homogeneous components of the grading are non-zero.
Using the same method as in \cite[Remark 3.2]{bag}, one may construct a counterexample with a trivial grading, i.e. with $S=S_e$.
\end{rem}

\section{Examples: Morita rings}\label{Sec:MoritaRing}

Let $(A,B, _AM_B, _BN_A, \varphi, \phi)$ be a strict Morita context.
It consists of unital rings $A$ and $B$,
an $A-B$-bimodule $M$, an $B-A$-bimodule $N$,
an $A-A$-bimodule epimorphism $\varphi : M \otimes_B N \to A$
and an $B-B$-bimodule epimorphism $\phi : N \otimes_A M \to B$.

The associated \emph{Morita ring} 
is the set
\begin{displaymath}
	S =
	\left(
	\begin{array}{cc}
		A & M \\
		N & B
	\end{array}
	\right)
\end{displaymath}
equipped with the natural addition and with a multiplication defined by
\begin{displaymath}
		\left(
	\begin{array}{cc}
		a_1 & m_1 \\
		n_1 & b_1
	\end{array}
	\right)
	*
		\left(
	\begin{array}{cc}
		a_2 & m_2 \\
		n_2 & b_2
	\end{array}
	\right)
	=
		\left(
	\begin{array}{cc}
		a_1a_2 + \varphi(m_1 \otimes n_2) & a_1m_2 + m_1 b_2\\
		n_1a_2 + b_1n_2 & \phi(n_1 \otimes m_2) + b_1b_2
	\end{array}
	\right)
\end{displaymath}
for $a_1,a_2\in A$, $b_1,b_2\in B$, $m_1,m_2\in M$ and $n_1,n_2\in N$.
Let $G$ be an infinite cyclic group, generated by $g$.
We can define a $G$-grading on $S$ by putting
\begin{displaymath}
	R=S_e =
	\left(
	\begin{array}{cc}
		A & 0 \\
		0 & B
	\end{array}
	\right),
	\quad
		S_g =
	\left(
	\begin{array}{cc}
		0 & M \\
		0 & 0
	\end{array}
	\right),
\quad	
		S_{g^{-1}} =
	\left(
	\begin{array}{cc}
		0 & 0 \\
		N & 0
	\end{array}
	\right)
\end{displaymath}
and $S_h=\left\{\left(\begin{smallmatrix}0&0\\0&0\end{smallmatrix}\right)\right\}$
for every $h \in G \setminus \{e,g,g^{-1}\}$.
It is easy to see that
\begin{displaymath}
	S_g S_{g^{-1}}
	=
		\left(
	\begin{array}{cc}
		\image(\varphi) & 0 \\
		0 & 0
	\end{array}
	\right) =\left(
	\begin{array}{cc}
		A & 0 \\
		0 & 0
	\end{array}
	\right)
\end{displaymath}
and
\begin{displaymath}
S_{g^{-1}} S_g 
	=
		\left(
	\begin{array}{cc}
		0 & 0 \\
		0 & \image(\phi) 
	\end{array}
	\right)	=\left(
	\begin{array}{cc}
		0  & 0 \\
		0 & B
	\end{array}
	\right)
\end{displaymath}
and thus $S$ is obviously not strongly graded.
However, $S$ is epsilon-strongly graded.
Indeed, if we put
\begin{displaymath}
	\epsilon_g
	=
		\left(
	\begin{array}{cc}
		1_A & 0 \\
		0 & 0
	\end{array}
	\right),
	\quad
\epsilon_{g^{-1}} 
	=
		\left(
	\begin{array}{cc}
		0 & 0 \\
		0 & 1_B
	\end{array}
	\right)	
	\text{ and }
	\epsilon_e = \epsilon_g + \epsilon_{g^{-1}}
\end{displaymath}
then it is easy to verify that this yields an epsilon-strong $G$-grading on $S$.
From the fact that
$\Supp(S)=\{g\in G \mid S_g \neq \{0\} \}$
is finite, we immediately see that
\begin{displaymath}
	Z(R)_{\rm fin} = Z(R) =
	\left(
	\begin{array}{cc}
		Z(A) & 0 \\
		0 &	Z(B)
	\end{array}
	\right).
\end{displaymath}
\begin{rem}
With this grading one can find examples in which the Morita ring
$S$ is not
a partial crossed product of $G$ over $R=S_e$. Indeed, let $P$ be a progenerator in the category ${\rm mod}-R,$ of right $R$-modules. It follows by \cite[Theorem 3.20]{Jacobson1989}  that $({\rm End} P_R, R, P, P^*={\rm hom}(P_R,R), \varphi, \phi),$ where $ \varphi\colon P^*\otimes_{{\rm End} P_R} P\ni f\otimes p\to f(p)\in R$ and $ \varphi\colon P\otimes_{R} P^*\ni p\otimes f\to f_p\in {\rm End} P_R,$ and $f_p(r)=pf(r),$ for all $r\in R,$ is a strict Morita context. Consider the associated Morita ring $S,$ and put $\mathcal{D}_g= S_g S_{g^{-1}}.$ Since, in general, as left ${\rm End} P_R$-modules, ${\rm End} P_R$ is not isomorphic to $P$, it follows by \cite[Theorem 6.5]{dokuchaev2008} that $S$ is not a partial crossed product of $G$ over $S_e.$

A concrete example is obtained by taking a commutative unital ring $R$ and $P=R^n,$ for  some $n>1.$\end{rem}

\begin{prop}\label{Prop:MoritaRingSeparability}
Let $(A,B, _AM_B, _BN_A, \varphi, \phi)$ be a strict Morita context, and let $S$ be the associated Morita ring. Then the extension $S/R$ is separable.
\end{prop}
\begin{proof} We know that $\gamma_0 : Z(R) \rightarrow Z(R)$ is the identity map on $Z(R)$.
Now we determine $\gamma_g$ and $\gamma_{g^{-1}}.$ Let $\sum_i m_i\otimes n_i\in M\otimes_B N$ be such that $\sum_i\varphi( m_i\otimes n_i)=1_A$ and  $\sum_j n_j\otimes m_j\in N\otimes_A M$ with $\sum_j\phi( n_j\otimes m_j)=1_B.$ Then for $ae_{11}+be_{22}\in Z(R)$ one has that
$$
\gamma_g(ae_{11}+be_{22})=\sum_i\varphi( m_ib\otimes n_i)e_{11}
$$
and
$$\gamma_{g^{-1}}(ae_{11}+be_{22})=\sum_j\phi( n_ja\otimes m_j)e_{22}.
$$
Then the trace map ${\rm tr}_{\gamma} : Z(R) \rightarrow Z(R)$ is given by
\begin{displaymath}
ae_{11}+be_{22}\mapsto  ae_{11}+be_{22} + \sum_i\varphi( m_ib\otimes n_i)e_{11}+ \sum_j\phi( n_ja\otimes m_j)e_{22}.
\end{displaymath}
From this it follows that ${\rm tr}(1_A e_{11})=1_A e_{11}+ 1_B e_{22},$ and hence $S/R$ is separable due to Theorem~\ref{maintheorem}.
\end{proof}

\begin{exa}\label{ex:epsfromstrong}
Let $T=\oplus_{g\in G} T_g$ be a ring which is strongly graded by a group $G$.
Fix $g\in G$ and consider the strict Morita context $(T_e,T_e,T_g,T_{g^{-1}},\varphi,\phi)$
where $\varphi : T_g \otimes_{T_e} T_{g^{-1}} \to T_e$ and $\phi : T_{g^{-1}} \otimes_{T_e} T_g \to T_e$
are the canonical $T_e$-bimodule isomorphisms (see \cite[Corollary 3.1.2]{nas04} and \cite{daa}).
The corresponding 
Morita ring 
$S =\left(\begin{smallmatrix}
		T_e & T_g \\
		T_{g^{-1}} & T_e
\end{smallmatrix}\right)$
is epsilon-strongly graded by an infinite cyclic group $G$,
generated by $g$, as described above.
By Proposition~\ref{Prop:MoritaRingSeparability}, $S$ 
is separable over
$R=\left(\begin{smallmatrix}T_e&0\\0&T_e\end{smallmatrix}\right)$.
\end{exa}

\begin{rem}
If $S$ is a ring which is strongly graded by $G$,
then $G=\Supp(S)=\{g\in G \mid S_g \neq \{0\} \}$ necessarily holds.
However, if $S$ is only epsilon-strongly graded by $G$,
then $\Supp(S)$ need not even be a subgroup of $G$.
Indeed, consider Example~\ref{ex:epsfromstrong} and notice that $g$ belongs to $\Supp(S)$
but that
$S_{g^2}=\left\{\left(\begin{smallmatrix}0&0\\0&0\end{smallmatrix}\right)\right\}$.
Hence, in this case $\Supp(S)$ is not closed under group multiplication.
\end{rem}

\section*{acknowledgement}
The authors are grateful to Ruy Exel for having pointed out the equivalence between (ii) and (iii) in Proposition~\ref{epsilon1}.

\end{document}